\documentclass[11pt,a4paper]{amsart}
\usepackage{amssymb}
\usepackage{amsthm}
\usepackage{amsmath}
\usepackage{amsxtra}
\usepackage{latexsym}
\usepackage{mathrsfs}
\usepackage[all,cmtip]{xy}
\usepackage[all]{xy}
\usepackage{enumitem}
\usepackage{xcolor}
\usepackage{graphicx}
\usepackage{comment}
\usepackage{mathabx,epsfig}
\usepackage{stmaryrd}

\usepackage{hyperref}

\newtheorem{thm}{Theorem}[section]
\newtheorem{lem}[thm]{Lemma}
\newtheorem{conj}[thm]{Conjecture}

\newtheorem{prop}[thm]{Proposition}
\newtheorem{cor}[thm]{Corollary}

\theoremstyle{definition}

\newtheorem{rmk}[thm]{Remark}

\numberwithin{equation}{section}

\newcommand\be{\begin{equation}}
\newcommand\ba{\begin{eqnarray}}
\newcommand\ee{\end{equation}}
\newcommand\ea{\end{eqnarray}}

  {\end{list}}

\title[]
{$p$-Adic interpolation of orbits under rational maps}
\author{Jason P.~Bell}
\address{University of Waterloo \\
Department of Pure Mathematics \\
Waterloo, Ontario \\
Canada  N2L 3G1}

\email{jpbell@uwaterloo.ca}
\thanks{The authors were supported by NSERC grant RGPIN-2016-03632.}
\author{Xiao Zhong}
\address{University of Waterloo \\
Department of Pure Mathematics \\
Waterloo, Ontario \\
Canada  N2L 3G1}
\email{x48zhong@uwaterloo.ca}


\begin{document}

\maketitle
\begin{abstract}
     Let $L$ be a field of characteristic zero, let $h:\mathbb{P}^1\to \mathbb{P}^1$ be a rational map defined over $L$, and let $c\in \mathbb{P}^1(L)$. We show that there exists a finitely generated subfield $K$ of $L$ over which both $c$ and $h$ are defined along with an infinite set of inequivalent non-archimedean completions $K_{\mathfrak{p}}$ for which
there exists a positive integer $a=a(\mathfrak{p})$ with the property that for $i\in \{0,\ldots ,a-1\}$ there exists a power series $g_i(t)\in K_{\mathfrak{p}}[[t]]$ that converges on the closed unit disc of $K_{\mathfrak{p}}$ such that $h^{an+i}(c)=g_i(n)$ for all sufficiently large $n$. As a consequence we show that the dynamical Mordell-Lang conjecture holds for split self-maps $(h,g)$ of $\mathbb{P}^1 \times X$ with $g$ \'etale.  
\end{abstract}

\section{Introduction}
The study of dynamics of rational functions in one complex variable enjoys a long and celebrated history, which stretches back to classical work of Fatou, Julia and others \cite{fat1, fat2, jul}. In this setting, one has a field $K$ and a rational function $f\in K(x)$, which one regards as a regular self-map from the projective line over $K$ to itself, and one then asks questions about the orbits of points under iteration of the map $f$. 

A key tool in understanding the dynamics of rational maps on $\mathbb{P}^1$ is by studying their periodic points (i.e., fixed points of given iterates) and their respective basins of attraction. Given a fixed point of a map, an important method when analyzing the behaviour of nearby points comes from analytic uniformization techniques.  In the case of $p$-adic maps there is an appealing trichotomy due to Rivera-Letelier \cite{RL}, which builds upon earlier work of Herman and Yoccoz \cite{HY}, which says that if $p$ is prime, $\mathbb{C}_p$ is the completion of an algebraic closure of $\mathbb{Q}_p$, and $f(z)=\lambda z+ \sum_{i\ge 2} a_i z^i \in \mathbb{C}_p[[z]]$ is a nonzero power series with $|\lambda|_p, |a_i|_p\le 1$ for all $i$, then $f$ has three possible types of analytic uniformization, which are dictated by whether the map $f$ is indifferent, attracting, or superattracting near the fixed point $z=0$.  
More precisely, Rivera-Letelier \cite{RL} shows the following trichotomy holds.
\begin{enumerate}
 \item[(a)] (\emph{Indifferent case}) If $|\lambda|_p=1$, then for $c\in \mathbb{C}_p$ sufficiently close to zero there exist an integer $a>0$ and power series $u_0(z),\ldots ,u_{a-1}(z)$ that converge on the unit disc such that $f^{an+i}(c)=u_i(n)$. 
    \item[(b)] (\emph{Attracting case}) If $0<|\lambda|_p<1$, then there is $r\in (0,1)$ and a power series $u(z)\in \mathbb{C}_p[[z]]$ that maps the closed disc $\overline{B(0,r)}$ bijectively to itself such that $f^n(z)=u(\lambda^n u^{-1}(z))$;
    \item[(c)] (\emph{Superattracting case}) If $\lambda=0$ and then there is some $m\ge 2$ and some $u(z)$ that bijectively maps a disc $\overline{B(0,r_1)}$ to another disc $\overline{B(0,r_2)}$ for some $r_1,r_2>0$ such that $f^n(z)=u((u^{-1}(z))^{m^n})$.
\end{enumerate}

Rivera-Letelier's characterization of possible analytic uniformizations of $p$-adic analytic maps has played an important role within arithmetic dynamics over the past fifteen years. It should be noted, however, that for many number theoretic questions concerning orbits of self-maps, it is much more desirable to work with maps that fall into case (a) of Rivera-Leterlier's trichotomy.  The reason for this is that in these cases one obtains a natural $p$-adic interpolation of orbits of points near the fixed point and so one can use tools from $p$-adic analysis to answer questions about points in the orbit.  On the other hand, case (c) only really says something about the rate of convergence of points in the orbit to the fixed point and it can be more difficult to glean number theoretic information from this information in practice.  

If one restricts one's focus to a fixed prime $p$ and looks at $p$-adic dynamics then cases (b) and (c) are at times unavoidable, but if one is only interested in the orbit of a point $c$ in a characteristic zero field $K$ under a rational map $f\in K(x)$, then in practice one can work instead with a finitely generated extension $K_0$ of $\mathbb{Q}$ inside $K$ over which $f$ and $c$ are both defined and one can then try to find a favourable prime $p$ and an embedding of $K_0$ into $\mathbb{C}_p$ so that case (a) of Rivera-Letelier's trichotomy applies to the orbit of the image of the point $c$.  In fact, we are able to show that there are infinitely many such primes for which one can embed $K_0$ into $\mathbb{C}_p$ such that after replacing $f$ by a suitable iterate we can $p$-adically interpolate the orbit of the image of $c$ with a $p$-adic analytic map. 
  \begin{thm} Let $L$ be a field of characteristic zero and let $h:\mathbb{P}^1\to \mathbb{P}^1$ be a rational map defined over $L$ and let $c\in \mathbb{P}^1(L)$.  Then there exists a finitely generated extension $K$ of $\mathbb{Q}$ over which both $c$ and $h$ are defined along with an infinite set of inequivalent non-archimedean completions $K_{\mathfrak{p}}$ such that 
there exists a positive integer $a=a(\mathfrak{p})$ with the property that for $i\in \{0,\ldots ,a-1\}$ there exists a power series $h_i(t)\in K_{\mathfrak{p}}[[t]]$ that converges on the closed unit disc of $K_{\mathfrak{p}}$ such that $h^{an+i}(c)=h_i(n)$ for all sufficiently large $n$.
\label{thm:main}
 \end{thm}
 
 We in fact prove a stronger result than the one given in the statement of Theorem \ref{thm:main} (see Remark \ref{rmk33}).  We also mention that related interpolation results appear in \cite[\S4]{BGHKST}, which deals with the case of split self-maps of $(\mathbb{P}^1)^m$.  Due to the more general setting considered by the authors in \cite{BGHKST}, stronger conditions on the maps involving critical points being preperiodic are necessarily imposed and it does not seem possible to obtain Theorem \ref{thm:main} in its full generality from these related results. 
 
In general, we cannot expect to do better than interpolating tails of orbits along progressions, since there are points whose orbits are preperiodic under certain rational maps and an analytic map that is constant on an infinite set of the closed $p$-adic ball is necessarily identically constant by Strassman's theorem.  The progressions arise in the proof of Theorem \ref{thm:main} since in order to apply interpolation results we must first replace 
$h$ by a suitable iterate $h^a$ and replace the starting point $c$ with $h^m(c)$ for some $m\ge 0$.  For this reason, one cannot eliminate the dependence of the integer $a$ on $\mathfrak{p}$ in the statement of Theorem \ref{thm:main}.  

Once one can interpolate the orbit $\{h^{an}(c_m)\}_{n\ge 0}$ with an analytic map as in the statement of Theorem \ref{thm:main}, one obtains an interpolation for $h^{an+i}(c_m)$ for $b\in \{0,\ldots ,a-1\}$ by applying the rational map $h^b$ to the analytic map interpolating $\{h^{an}(c_m)\}_{n\ge 0}$.

While Rivera-Letelier's results give a strong and very useful trichotomy for studying the dynamics of rational maps over a $p$-adic field, for many arithmetical applications it is often more desirable to have a map in which the dynamics fall into the first two cases of Rivera-Letelier's trichotomy, since one can use parametrization of the orbit by an analytic map to draw conclusions about the map (this idea was apparently first applied by Skolem \cite{Sko}, and has since been use in many other works \cite{Am2, Am, BGT10, BGT2, BSS, BGHKST, BGKT, CX, GX, LL}).  If one works over a fixed $p$-adic field, however, then one cannot guarantee that the orbit is covered under the third case of Rivera-Letelier's trichotomy.  We show that after working over a suitable finitely generated extension $K$ of $\mathbb{Q}$, over which our point and self-map are both defined, we can find many non-archimedean completions of $K$ for which we obtain an analytic interpolation of our orbit as in the statement of Theorem \ref{thm:main}.

In recent years, one of the most important applications of $p$-adic interpolation techniques has been to settle cases of the so-called dynamical Mordell-Lang conjecture, which can be viewed as a natural dynamical analogue of the cyclic case of the classical Mordell-Lang conjecture and was first formulated in \cite{GT}. The classical Mordell-Lang conjecture was settled in a series of works by Faltings \cite{Fal}, Vojta \cite{Voj}, and McQuillan \cite{McQ}.

\begin{conj}{(The dynamical Mordell-Lang Conjecture)}
Let $X$ be a complex quasiprojective variety and let $\Phi$ be a rational self-map of $X$.  Given $c\in X$ with the property that the forward orbit of $c$ under $\Phi$ avoids the indeterminacy locus of $\Phi$ and a Zariski closed subset $Y\subseteq X$, the set of $n \in \mathbb{N}$ such that $\Phi^n(c) \in Y$ is a union of finitely many infinite arithmetic progressions augmented by a finite set.
\label{conj:DML}
\end{conj}
We note that it is possible to have an empty union of infinite arithmetic progressions in Conjecture \ref{conj:DML}, in which case the orbit of $c$ has finite intersection with $Y$; it is also possible for the finite set to be empty.

Conjecture \ref{conj:DML} is known in several cases, including when $\Phi$ is \'etale, when $X=\mathbb{A}^2$ and $\Phi$ is respectively an endomorphism \cite{Xie1} and a birational self-map \cite{Xie2}, and in several other cases \cite{BGKT, GX}. In the case when $\Phi$ is \'etale, the orbit of a point has many $p$-adic parametrizations along progressions, and this fact, combined with Theorem \ref{thm:main}, allows us to deduce that Conjecture \ref{conj:DML} holds for split maps $(h,g)$ of $\mathbb{P}^1\times X$ with $g$ \'etale.

\begin{cor} Let $X$ be a complex quasiprojective variety and let $g:X\to X$ be an \'etale self-map of $X$ and let $h\in \mathbb{C}(x)$.  If $\Phi=(h,g): \mathbb{P}^1\times X\to \mathbb{P}^1\times X$, $c\in \mathbb{P}^1\times X$, and $Y\subseteq \mathbb{P}^1\times X$ is a Zariski closed subset, then $\{n\colon \Phi^n(c)\in Y\}$ is a finite union of infinite arithmetic progressions along with a finite set.
\label{cor:main}
\end{cor}
The outline of this paper is as follows. In \S\ref{inter}, we give general interpolation results, which apply to rational maps having at least four distinct non-superattracting fixed points.  In \S\ref{thm}, we prove Theorem \ref{thm:main} by reducing to the case considered in \S\ref{inter}. In \S\ref{DML} we prove Corollary \ref{cor:main} and in \S\ref{conc} we make some concluding remarks, which show the difficulties with trying to extend our interpolation results to higher dimensions.  
 \section{$p$-Adic interpolation}\label{inter}
 In this section we prove a special case of Theorem \ref{thm:main}.  We let $L$ be a field of characteristic zero and we let $h(x)\in L(x)$, and we let $c\in L$.  We wish to understand the orbit of $c$ under the self-map $h$.  
 
 As it turns out, the difficult case is when $h(x)$ has degree at least two. Then by a result of Fatou \cite{fat1,fat2,jul} we know that $h$ has at most $2{\rm deg}(h)-2$ periodic superattracting cycles (i.e., orbits of points $a\in \bar{L}$ with the property that $h^m(a)=a$ and $(h^m)'(a)=0$) and since $h$ has a Zariski dense set of periodic points, some iterate of $h$ will have at least four non-superattracting fixed points.  Then by the remarks after the statement of Theorem \ref{thm:main} show, we may replace $h$ by an iterate and assume that it has at least four non-superattracting fixed points.  Moreover, we can conjugate $h$ by a suitable fractional linear transformation and assume that $\infty$ is fixed by $h$ and is not superattracting, which means that the degree of the numerator of $h$ is one greater than the degree of the denominator.  
   
 Throughout this section, we will thus assume that the above remarks apply to $h(x)$ and we write

 \be
 h(x) = p(x)/q(x), \gcd(p(x),q(x))=1,~{\rm deg}(p)=1+{\rm deg}(q).
 \label{eq:1}
  \ee

By factoring over the algebraic closure of $K$ we then have
 \be 
 p(x)  = C(x- \alpha_1)^{a_1}\cdots (x - \alpha_s)^{a_s},
  \label{eq:3}
 \ee 
 \be 
 q(x) = (x - \beta_1)^{b_1}\cdots (x - \beta_{t})^{b_{t}},
  \label{eq:4}
 \ee 
 \be 
 p(x) - xq(x) = C' (x - \gamma_1)^{c_1} \cdots (x - \gamma_u)^{c_u},
  \label{eq:5}
 \ee 
 \be
  p'(x)q(x) - q'(x)p(x) = C'' (x - \delta_1)^{d_1} \cdots (x - \delta_v)^{d_v},
   \label{eq:6}
 \ee 
 where 
 \begin{equation}
 \mathcal{T}:=\{ \alpha_1,\ldots, \alpha_s, \beta_1,\ldots, \beta_{t}, \gamma_1,\ldots, \gamma_u,\delta_1,\ldots, \delta_v\}\subseteq \overline{K},\end{equation}
 and $\alpha_1,\ldots ,\alpha_s,\beta_1,\ldots ,\beta_{t}$ are pairwise distinct, $C,C',C''$ are nonzero, $\gamma_1,\ldots ,\gamma_u$ are pairwise distinct, and $\delta_1,\ldots ,\delta_v$ are pairwise distinct.  
 
 Since $h(x)$ has at least three non-superattracting fixed points in $\mathbb{P}^1\setminus \{\infty\}$, the number of distinct roots of $p(x)-x q(x)$, is at least $3$, and we may assume that $\gamma_1,\gamma_2,\gamma_3$ are disjoint from $\{\delta_1,\ldots ,\delta_v\}$.  
 Since $h(x)$ is non-constant we also have $v>0$.   

 Consider the ring \be
 \label{eq:7}R := \mathbb{Z}[c, 6^{-1}, C^{\pm 1}, (C')^{\pm 1}, (C'')^{\pm 1}, \mathcal{T}][\mathcal{S}^{-1}],\ee
 where $\mathcal{S}$ is the union of the nonzero elements of $\mathcal{T}$ along with the set of elements that can be expressed as a difference of two distinct elements from $\mathcal{T}$.  Inverting $6$ is not technically necessary, but we do this for convenience as the $p$-adic arguments that we will use are slightly cleaner for primes larger than $3$.

 We now give a description of the orbit of the point $c$ under the map $h$ as a fraction of elements of $R$.  We let 
 \be A_0=c,~ B_0=1.\ee  Then for each $n$ we will give coprime elements $A_n,B_n$ of $R$ with the property that
 $A_n/B_n = h^n(c)$.  
 To do this, for $n\ge 0$, we define
 \begin{equation}
 A_{n+1} = C(A_n-\alpha_1 B_n)^{a_1}\cdots (A_n-\alpha_s B_n)^{a_s} = p(A_n/B_n) \cdot B_n^{{\rm deg}(p(x))}
 \label{eq:8}
 \end{equation}
 and
  \begin{equation}
   \label{eq:9}
 B_{n+1} = B_n (A_n-\beta_1 B_n)^{b_1}\cdots (A_n-\beta_t B_n)^{b_t} = q(A_n/B_n) \cdot B_n^{{\rm deg}(p(x))}.
 \end{equation}
 Observe that 
\begin{align} &= B_n A_{n+1} - A_n B_{n+1} \nonumber \\ &= 
 B_n B_{n+1} (h(A_n/B_n) - A_n/B_n)\nonumber \\
 &= B_n B_{n+1} ( p(A_n/B_n) - q(A_n/B_n) A_n/B_n) q(A_n/B_n)^{-1}\nonumber \\
 &= C' B_n B_{n+1} (A_n - \gamma_1 B_n)^{c_1} \cdots (A_n - \gamma_u B_n)^{c_u} B_n^{-c_1-\cdots -c_u} B_{n+1}^{-1} B_n^{{\rm deg}(p)}\nonumber
 \end{align}
  Since ${\rm deg}(p(x)-xq(x))=c_1+\cdots + c_u \le {\rm deg}(p(x))$, there is some $\ell_n\ge 1$ such that 
  \be \label{eq:10}
B_n A_{n+1} - A_n B_{n+1}  =C' B_n^{\ell_n} (A_n - \gamma_1 B_n)^{c_1} \cdots (A_n - \gamma_u B_n)^{c_u}\ee 

 \begin{lem} Adopt the notation from Equations (\ref{eq:1})--(\ref{eq:9}).
Then for each $n\ge 0$, the elements $A_n$ and $B_n$ generate the unit ideal in $R$.
\label{lem:unit}
 \end{lem}
 \begin{proof}
 We prove this by induction on $n$.  When $n=0$ it is immediate, since $B_0=1$.  Now suppose that the result holds whenever $n\le m$ with $m\ge 1$ and consider the case when $n=m+1$.
If $A_{m+1}R+B_{m+1}R\neq R$, there is a maximal ideal $\mathfrak{p}$ that contains both $A_{m+1}$ and $B_{m+1}$. 
 Then since $C$ is a unit in $R$, we must have $A_m-\alpha_i B_m\in \mathfrak{p}$ for some $i$ and either $B_n\in \mathfrak{p}$ or $A_m-\beta_j B_m\in\mathfrak{p}$ for some $j$.  
 If $B_m\in \mathfrak{p}$ then we see $A_m = (A_m-\alpha_i B_m)+\alpha_i B_m\in \mathfrak{p}$, which is impossible by our induction hypothesis.  Thus we may assume that $B_m\not\in \mathfrak{p}$ and $A_m-\beta_j B_m\in\mathfrak{p}$ for some $j$.  
 Since $p(x)$ and $q(x)$ are coprime, $\{\alpha_1,\ldots , \alpha_s\}\cap \{\beta_1,\ldots ,\beta_t\}$ is empty, and so if
 $A_m-\alpha_i B_m\in\mathfrak{p}$ and $A_n-\beta_j B_n\in\mathfrak{p}$, then $(\alpha_i-\beta_j)B_j\in \mathfrak{p}$, which is again impossible since $\alpha_i-\beta_j$ is a unit in $R$. Thus we obtain the desired result.  
 \end{proof}
 
  We now prove a useful lemma, in which we make use of the $S$-unit theorem (see \cite[Theorem 6.1.3]{EG}).  We recall that if $K$ is a field of characteristic zero and $G\le K^*$ is a finitely generated subgroup of the multiplicative group, then the $S$-unit theorem concerns solutions $(X_1,\ldots ,X_n)\in G^n$ to 
the equation $$\sum_{i=1}^n \rho_i X_i = 0$$ with $\rho_1,\ldots ,\rho_n$ fixed nonzero elements of $K$.  We say that $(X_1,\ldots ,X_n)$ is a \emph{non-degenerate} solution if no proper non-trivial subsum of $\sum \rho_i X_i$ vanishes.  Then the $S$-unit theorem says that up to scaling by elements of $G$ there are only finitely many non-degenerate solutions $(X_1,\ldots ,X_n)\in G^n$ to the above equation.  We recall that a subset $\mathcal{P}$ of the prime spectrum of a commutative ring is Zariski dense if the intersection of the prime ideals in $\mathcal{P}$ is contained in the nil radical of the ring. 

 \begin{lem} \label{l2} Adopt the notation from Equations (\ref{eq:1})--(\ref{eq:9}), and suppose that $u\ge 3$ and that $\gamma_1,\gamma_2, \gamma_3\not \in \{\delta_1,\ldots ,\delta_v\}$.
 If the sequence $\{A_n/B_n\}$ is not eventually periodic then there is a Zariski dense set of maximal ideals $\mathcal{P}$ of $R$ such that for each $\mathfrak{p}\in \mathcal{P}$ there is some natural number $n$ such that the following hold:
\begin{enumerate}
\item $A_n B_{n+1}-B_n A_{n+1}\in \mathfrak{p}$;
\item $B_n\cdot B_{n+1}\not\in \mathfrak{p}$; and 
\item $(p'(A_n/B_n) q(A_n/B_n) - q'(A_n/B_n) p(A_n/B_n))B_n^{{\rm deg}(p(x))+{\rm deg}(q(x))-1}$ is not in $\mathfrak{p}$.
\end{enumerate}
 \end{lem}
 \begin{proof}
 We may assume that the sequence $\{A_n/B_n\}$ is not eventually periodic.
 Let $\mathcal{P}$ denote the set of maximal ideals $\mathfrak{p}$ for which conditions (1)--(3) hold.  If $\mathcal{P}$ is not Zariski dense, then there is some nonzero $f\in R$ such that $f$ is in every prime in $\mathcal{P}$.  Then after replacing $R$ by $R[1/f]$, we may assume that $\mathcal{P}$ is empty.  Let $U$ denote the group of units of $R$, which is a finitely generated abelian group by Roquette's theorem \cite{Roq}.
 
 By assumption $\gamma_1,\gamma_2, \gamma_3\not \in \{\delta_1,\ldots ,\delta_v\}$.
 Suppose that $A_n-\gamma_1 B_n$, $A_n-\gamma_2 B_n$, $A_n-\gamma_3 B_n$ are all in $U$ for every $n$.  Then we pick $\rho_1, \rho_2, \rho_3 \in R$, not all zero, such that 
 $\rho_1+\rho_2+\rho_3=\gamma_1 \rho_1+\gamma_2 \rho_2+\gamma_3 \rho_3=0$. 
 Since $\gamma_1, \gamma_2, \gamma_3$ are pairwise distinct, we in fact have $\rho_1,\rho_2,\rho_3$ are all nonzero.
 
   Then by construction
 $$\rho_1 (A_n-\gamma_1 B_n) + \rho_2 (A_n-\gamma_2 B_n) +\rho_3 (A_n-\gamma_3 B_n) = 0$$ for every $n\ge 0$.  Since every solution to 
 $\rho_1 x_1+\rho_2 x_2+\rho_3 x_3 = 0$ with $x_1x_2x_3\neq 0$ is necessarily non-degenerate, by the $S$-unit theorem there are only finitely many solutions $(x_1,x_2,x_3)$ in $U^3$ up to scaling.  
 
 It follows that there must exist $n$ and $m$ with $n<m$ such that
 $(A_n-\gamma_1 B_n)/(A_n-\gamma_2 B_n) = (A_m-\gamma_1 B_m)/(A_m-\gamma_2 B_m)$.  In other words, $\phi(A_n/B_n)=\phi(A_m/B_m)$, where $\phi(x)=(x-\gamma_1)/(x-\gamma_2)$.  Since $\phi$ is an automorphism of $\mathbb{P}^1$, we then conclude that $A_n/B_n=A_m/B_m$ and so the sequence $\{A_i/B_i\}$ is eventually periodic, which is a contradiction.

 It follows that there is some $n$ such that $A_n-\gamma_i B_n\not\in U$ for some $i\in \{1,2,3\}$.  Then $A_n-\gamma_i B_n\neq 0$ since otherwise, we would have $A_{n+1}/B_{n+1}= A_n/B_n$ and so we conclude that there is some maximal ideal $\mathfrak{p}$ of $R$ such that $A_n-\gamma_i B_n\in \mathfrak{p}$.  In particular, $B_n A_{n+1}-A_n B_{n+1}\in \mathfrak{p}$.  
 
 If $B_n\in \mathfrak{p}$ then $A_n\in \mathfrak{p}$ since $A_n-\gamma_i B_n\in \mathfrak{p}$, and this is impossible by Lemma \ref{lem:unit}.
 
 Now if $B_{n+1}\in \mathfrak{p}$ then since $$B_n A_{n+1}-A_n B_{n+1}\in\mathfrak{p}$$ and $B_n\not\in \mathfrak{p}$, we see that $A_{n+1}\in\mathfrak{p}$, which contradicts the conclusion of Lemma \ref{lem:unit}.
 
 Finally, if $$p'(A_n/B_n) q(A_n/B_n) - q'(A_n/B_n) p(A_n/B_n))B_n^{{\rm deg}(p(x))+{\rm deg}(q(x))-1}$$ is in $\mathfrak{p}$ then since $B_n$ is not in $\mathfrak{p}$ and since $C''$ is a unit, by Equation (\ref{eq:10}) we must have
 $A_n -\delta_j B_n\in \mathfrak{p}$ for some $j$, and so $(\delta_j-\gamma_i)B_n$ is in $\mathfrak{p}$.  But since $\gamma_i\neq \delta_j$, we have that $\gamma_i-\delta_j$ is a unit in $R$ by construction and since $B_n$ is not in $\mathfrak{p}$, we then see that this cannot hold. 
 The result follows.   
  \end{proof}
We will now obtain our interpolation of the orbit of $c$ under $h$ by applying a result of Poonen \cite{poonen}.  We need a simple lemma, which will give us the hypotheses needed to apply Poonen's result.  We recall that if $K$ is a field with a non-archimedean absolute value $|~|$ and $R$ is a subring of $K$ then the \emph{Tate algebra} $R\langle x_1,\ldots ,x_d\rangle$ is the subset of $R[[x_1,\ldots ,x_d]]$ consisting of power series that converge on the unit polydisc of $K^d$.  
\begin{lem}\label{lem:o}
Let $\mathfrak{o}$ be a complete discrete valuation ring, suppose that there is a prime $\pi\in \mathbb{Z}$ with $\pi\ge 5$ such that $|\pi|<1$ in $\mathfrak{o}$, and suppose that $h(x)=p(x)/q(x)$ with $p(x),q(x)\in \mathfrak{o}[x]$.  Suppose further that $c\in \mathfrak{o}$ is such that:
\begin{enumerate}
\item $|q(c)|=1$;
\item $h(c)\equiv c~(\bmod \pi^2\mathfrak{o})$; and
\item $h'(c)\equiv 1~(\bmod ~\pi\mathfrak{o})$.
\end{enumerate} 
Then the map 
$f(x):=\pi^{-1}(h(c+\pi x)-c)$ is in $\mathfrak{o}\langle x\rangle$ and satisfies $f(x)\equiv x~(\bmod ~\pi\mathfrak{o})$ for all $x\in \mathfrak{o}$.  
\end{lem}
\begin{proof}  We have $q(c+\pi x) \equiv q(c)~(\bmod~\pi)$ and so $|q(c+\pi x)|=1$ for all $x\in \mathfrak{o}$.  It follows that $h^{(r)}(c+\pi x)\in \mathfrak{o}$ for all $r\ge 0$ and all $x\in \mathfrak{o}$.  
Then 
\begin{align*}
f(x) &= \pi^{-1}\left( h(c+\pi x)-c\right)\\
&= \pi^{-1}\left( \sum_{r\ge 0} h^{(r)}(c) \pi^{r} x^r/r! - c\right)\\
&= (h(c)-c)\pi^{-1} + h'(c)x + \sum_{r\ge 2} h^{(r)}(c) \frac{\pi^{r-1}}{r!} x^r.
\end{align*}
Since $\pi\neq 2,3$ and since $|r!|_{\pi} > \pi^{-r/(\pi-1)}$, we see that $\pi^{r-1}/r! \in \pi \mathfrak{o}$ for all $r\ge 2$, and that $|h^{(r)}(c) \frac{\pi^{r-1}}{r!} |\to 0$ as $r\to\infty$ and hence $f(x)\in \mathfrak{o}\langle x\rangle$.   Next, for $x\in \mathfrak{o}$, $$f(c+\pi x)- x = (h(c)-c)\pi^{-1} + (h'(c)-1)x+ \sum_{r\ge 2} h^{(r)}(c) \frac{\pi^{r-1}}{r!} x^r \equiv 0~(\bmod~\pi\mathfrak{o}),$$ by our assumptions and the remarks above. The result follows.
\end{proof}

\begin{prop}\label{prop:main}
Adopt the notation of Equations (\ref{eq:1})--(\ref{eq:10}).  Then there is a Zariski dense set of maximal ideals $\mathcal{P}$ of $R$ with the following properties:
\begin{enumerate}
\item[(i)]  each $\mathfrak{p}\in \mathcal{P}$ induces a non-archimedean absolute value $|~|_{\mathfrak{p}}$ on the field of fractions, $K_{\mathfrak{p}}$, of the completion of the local ring $R_{\mathfrak{p}}$;
\item[(ii)] for each $\mathfrak{p}\in \mathcal{P}$, there is a natural number $a$ and elements $$g_0(z),\ldots ,g_{a-1}(z)\in \mathfrak{o}\langle z\rangle$$ such that for each $i\in \{0,\ldots ,a-1\}$, $h^{an+i}(c)=g_i(n)$ for all $n$ sufficiently large, where $\mathfrak{o}$ is the valuation subring of $K_{\mathfrak{p}}$ consisting of elements $r$ with $|r|_{\mathfrak{p}}\le 1$.
\end{enumerate}
\end{prop}
\begin{proof}
By Lemma \ref{l2} there is a Zariski dense set $\mathcal{P}$ of maximal ideals $\mathfrak{p}$ such that items (1)--(3) from the statement of Lemma \ref{l2} hold.  
Since the regular locus of ${\rm Spec}(R)$ is a dense open set, we can replace $\mathcal{P}$ by a dense subset with the additional property that $R_{\mathfrak{p}}$ is a regular local noetherian ring for $\mathfrak{p}\in \mathcal{P}$.  

Since $R_{\mathfrak{p}}$ is regular, its associated graded ring $$\bigoplus_{n\ge 0} \mathfrak{p}^nR_{\mathfrak{p}}/\mathfrak{p}^{n+1}R_{\mathfrak{p}}$$ is a polynomial ring over the residue field $\mathbb{F}:=R_{\mathfrak{p}}/\mathfrak{p}R_{\mathfrak{p}}$; moreover, $\mathbb{F}$ is a finite field by the Nullstellensatz \cite[Theorem 4.19]{Eis}.

Since the associated graded ring of the local ring $R_{\mathfrak{p}}$ is an integral domain, we have an absolute value $|~|=|~|_{\mathfrak{p}}$ on $R_{\mathfrak{p}}$ given by $|0|=0$ and for $a$ nonzero, $|a|=|\mathbb{F}|^{-\nu(a)}$, where $\nu(a)$ is the largest nonnegative integer $r$ such that $a\in \mathfrak{p}^r R_{\mathfrak{p}}$.   Such an $r$ necessarily exists by the Krull intersection theorem.
  
Then this absolute value extends to an absolute value on $K_{\mathfrak{p}}$, where $K_{\mathfrak{p}}$ is the field of fractions of the completion of $R_{\mathfrak{p}}$.  We now let $\mathfrak{o}$ denote the valuation subring of $K_{\mathfrak{p}}$ consisting of elements of absolute value at most $1$.  Then $\mathfrak{o}$ is a discrete valuation ring and there is a unique prime $\pi \in \mathbb{Z}$ such that $|\pi | < 1$, since $\mathbb{F}$ is finite.

Let $c_i = A_i/B_i\in R_{\mathfrak{p}}$ for $i\ge 0$.  Then by assumption $$h(c_m)\equiv 
c_m~(\bmod~\mathfrak{p}R_{\mathfrak{p}})\qquad {\rm and}\qquad h'(c_m)\not\equiv 0~(\bmod~\mathfrak{p}
R_{\mathfrak{p}}).$$  

Moreover, by our choice of $\mathfrak{p}$ we have that $|B_n|=1$ for 
all $n\ge m$ and so $c_n\equiv c_m~(\bmod ~  \mathfrak{p}R_{\mathfrak{p}})$ for $n\ge 
m$.

It follows that there exist $m',m''\ge m$ with $m'>m''$ such that $c_{m'}\equiv c_{m''}
~(\bmod ~\pi^2\mathfrak{o})$ and $h'(c_{m'})\equiv 1~(\bmod ~\pi \mathfrak{o})$.  
Thus we see that after replacing $h$ by $h^{m'-m''}$, we may assume that the conditions in Lemma \ref{lem:o} are satisfied and so the map $$f(x) := \pi^{-1}(h(c_{m''}+\pi x)-c_{m''})$$
 satisfies $f(x)\equiv x~(\bmod~\pi \mathfrak{o})$.  Hence by a result of Poonen \cite{poonen} we have there is a map $g(x,n)\in \mathfrak{o}\langle x,n\rangle$ such that 
$f^n(x)=g(x,n)$.  In particular, $h^n(c_{m''}) = \pi f^n(0)+c_{m''} = \pi g(0,n) + c_{m''}$, and so we have obtained an interpolation of the orbit of $c_{m''}$ under $h$.  Since we did this at the expense of replacing $h$ by an iterate and replacing our starting point $c_0$ with a different point in the orbit, we see that this gives the desired result, as explained in the remarks following the statement of Theorem \ref{thm:main}.
\end{proof}
\section{Proof of Theorem \ref{thm:main}}\label{thm}
We now use the results of the preceding section to prove our main interpolation result.


 \begin{proof}[Proof of Theorem \ref{thm:main}]
 If $h$ is of degree one, then $h$ is \'etale and the result follows from BGT.  Similarly, if the orbit of $c$ under of iteration of $h$ is preperiodic then the result holds trivially.  Thus we may assume that the orbit is infinite and that $h$ has degree at least two.  By a result of Fatou \cite{fat1,fat2,jul}, there are at most $2{\rm deg}(h)-2$ attracting periodic cycles, and since the set of periodic points of $h$ is infinite, after replacing $h$ be an iterate, we may assume that $h$ has at least four non-attracting fixed points and by enlarging $K$ and conjugating $h$ by a fractional linear transformation, we may assume that $\infty$ is fixed by $h$.  (We note that if we can interpolate the orbit of a point $\phi(c)$ under a conjugate $\phi\circ h\circ\phi^{-1}$ of $h$, then we can interpolate the orbit of $c$ under $h$, by applying $\phi^{-1}$ to our interpolating power series.) Then by Proposition \ref{prop:main} we obtain the desired result in this case.
 \end{proof}
 We now make several remarks that are potentially useful for applications of Theorem \ref{thm:main}.
 \begin{rmk}
 We observe that an analogous conclusion to that of Theorem \ref{thm:main} can be obtained for self-maps of curves.  The reason for this is that after replacing the map by a suitable iterate, it suffices to consider the case of a geometrically irreducible curve.  We can also pass to the normalization and assume that our curve is smooth, and Theorem \ref{thm:main} then handles the genus $0$ case.  The genus one case follows from \cite{BGT10} and for curves of genus $\ge 2$, every endomorphism is an automorphism by the Riemann-Hurwitz formula and has finite order \cite[Ex. IV 2.5, IV 5.2, V 1.11]{Hartshorne}, and so the result holds trivially in this last case.
  \end{rmk}
 \begin{rmk} In some cases it is useful to keep track of additional geometric data and thus to enlarge the ring $R$ in Equation (\ref{eq:7}).  We note that a finite set of additional generators from the ambient field $L$ can be added to the ring $R$ without affecting the arguments.  
 \label{rmk32}
 \end{rmk}
 \begin{rmk}\label{rmk33}
 We in fact show something strictly stronger than merely having an infinite set of pairwise distinct completions of $K$ in Theorem \ref{thm:main}.  The proofs shows that there is a finitely generated subring $R$ of $K$ whose field of fractions is $K$ that is a subring of the valuation ring for each of our absolute values $|~|$ and the set of elements $r\in R$ such that $|r|<1$ for each of the absolute values we construct is $\{0\}$. 
 \end{rmk}
 \section{An instance of the dynamical Mordell-Lang Conjecture}\label{DML}
 In this section, we apply Theorem \ref{thm:main} to obtain an instance of the dynamical Mordell-Lang conjecture for split endomorphisms of a certain form.  We make use of the results of \cite{BGT10}, which is a precursor to the work of Poonen, and uses results about embedding finitely generated rings into $p$-adic rings rather than completions.  Nevertheless it is straightforward to translate these results to the framework we work with.
 
 \begin{proof}[Proof of Corollary \ref{cor:main}]
 We write $c=(c',c'')\in \mathbb{P}^1\times X$.
 
As in the proof of Theorem \ref{thm:main}, if $h$ has degree one, then $h$ is \'etale and we can infer the result directly from \cite{BGT10}.  Thus we may assume that $h$ has degree at least $2$ and by replacing $\Phi$ (and hence $h$) by an iterate and possibly conjugating $h$ by an automorphism of $\mathbb{P}^1$ (i.e., making a change of variables), we may assume that $h$ satisfies the hypotheses from \S\ref{inter} and in particular we now adopt the notation from Equations (\ref{eq:1})--(\ref{eq:10}).  As the remarks following the statement of Theorem \ref{thm:main} show, we can replace $\Phi$ by a suitable iterate and still obtain the desired result for the original map $\Phi$. 

It is sufficient to consider the case when $X$ is smooth and geometrically irreducible by the argument given in \cite[Theorem 1.3]{BGT10}. Thus we assume that $X$ is a irreducible and smooth quasiprojective variety.  Let $\rho: X \longrightarrow \mathbb{P}^M(\mathbb{C})$ be an embedding of $X$ into projective space.  Then \cite[Theorem 4.1]{BGT10} shows there is a finitely generated $\mathbb{Z}$-subalgebra $S$ of $\mathbb{C}$ for which we obtain a model $\mathcal{X}\subseteq \mathbb{P}^M_{{\rm Spec}(S)}$ of $X$ over ${\rm Spec}(S)$.   By adding a finite set of additional elements to $S$, we may assume that $S$ is a finitely generated $R$-algebra, where $R$ is as in Equation \ref{eq:7}.

Then \cite[Proposition 4.3]{BGT10} shows that there is a dense open subset $U$ of ${\rm Spec}(S)$ such that there is a scheme $\mathcal{X}_U$ that is smooth and quasiprojective over $U$ whose generic fibre is $X$ such that the endomorphism $g$ of $X$ extends to an unramified endomorphism $g_U$ of $\mathcal{X}_U$.  Since $U$ is a dense open subset of ${\rm Spec}(S)$, by Proposition \ref{prop:main} combined with Remark \ref{rmk32} we have some maximal ideal $\mathfrak{p}$ that is in $U$, where (i) and (ii) from Proposition \ref{prop:main} hold for the maximal ideal $\mathfrak{p}$.

In particular, if we let $K_{\mathfrak{p}}$ be the field of fractions of the completion of the local ring $S_{\mathfrak{p}}$ and let $\mathbb{F}$ denote the residue field $S_{\mathfrak{p}}/\mathfrak{p}S_{\mathfrak{p}}$, then there is a natural number $a$ and elements $$h_0(z),\ldots ,h_{a-1}(z)\in \mathfrak{o}\langle z\rangle$$ such that for each $i\in \{0,\ldots ,a-1\}$, $h^{an+i}(c')=h_i(n)$ for all $n$ sufficiently large, where $\mathfrak{o}$ is the valuation subring of $K_{\mathfrak{p}}$.

Since $K_{\mathfrak{p}}$ is the field of fractions of the completion of a finitely generated $\mathbb{Z}$-algebra with respect to a maximal ideal $\mathfrak{p}$, there is a unique prime $\ell$ such that it is isomorphic to a topological subfield of $\mathbb{C}_{\ell}$.  
Then the arguments of \cite{BGT10}\footnote{While the arguments are done over $\mathbb{Z}_p$ they work with any complete rank one discrete valuation ring of mixed characteristic with finite residue field.} show that if we regard $X(\mathbb{C}_{\ell})$ as a $d$-dimensional $\mathbb{C}_{\ell}$-manifold, then after replacing $c''$ by $g^m(c'')$ for some $m$, there is some integer $b\ge 1$ and analytic open neighbourhood $V$ of $g^m(c'')$ inside $X(K_{\mathfrak{p}})$ that is invariant under $g^b$ and an analytic bijection $\iota: V\to {\mathfrak{o}}^d$ such that 
$\iota \circ g^{bn}(g^m(c'')) = (g_1(n),\ldots ,g_d(n))$, where $g_1(z),\ldots ,g_d(z)\in\mathfrak{o}\langle z\rangle$.  

In particular, the arguments of \cite[Theorem 4.1]{BGT10} show that if we take $e$ to be the least common multiple of $a$ and $b$ then the set of sufficiently large positive integers in $\{n\colon \Phi^{en+i}(c)\in Y\}$ can be realized as the set of common positive integer zeros of a finite set of maps in $\mathfrak{o}\langle z\rangle$.  In particular, by Strassman's theorem we get that $$\{n\colon \Phi^{en+i}(c)\in Y\}$$ either contains all sufficiently large natural numbers $n$ or it is a finite set.  The result follows.
\end{proof}
 
  Occasionally one can use data from canonical heights or other methods to prove the dynamical Mordell-Lang for certain classes of endomorphisms.  Such methods do not seem to be generally applicable to the general case we consider.  As an example, consider $\mathbb{P}^1\times E$, where $E$ is an elliptic curve, and let $\Phi=(g,[2])$, where $[2]$ is multiplication by $2$ and $g$ is a rational map of degree four.  Then the heights of $g^n(a)$ and $[2^n]\cdot b$ for a point $(a,b)\in \mathbb{P}^1\times E$ with $a$ not preperiodic under $g$ and $b$ not a torsion point of $E$, are both asymptotic to a nonzero constant times $4^n$ as $n\to\infty$ and so there are non-periodic curves $C\subseteq \mathbb{P}^1\times E$ where one cannot naively use information from heights to rule out the orbit having infinite intersection with $C$.  
\section{Concluding remarks} \label{conc} 
We have shown that one can interpolate self-maps of curves, and it is natural to ask whether one can obtain similar interpolation results for endomorphisms of higher dimensional varieties.  Unfortunately, this fails even for surfaces.  As a simple example, consider the map $f:\mathbb{A}^2\to \mathbb{A}^2$ given by $f(x,y)=(x+1,y(x+1))$.  Then 
$f^n(0,1) = (n,n!)$ and so for each prime $p$ we have $|n!|_p \to 0$ and so there is no way to interpolate the orbit of $(0,1)$ under $f$.  

A more reasonable goal is to consider simultaneous interpolation for several rational maps $h_1,\ldots, h_m\in \mathbb{C}(x)$ to obtain the case of split rational maps from $\left(\mathbb{P}^1\right)^m$ to itself.  In this case, we do not know whether this can be done even when $m=2$.  Heuristics (see \cite[\S5]{BGHKST}) suggest, however, that one should be able to obtain the split case for rational maps under general conditions \cite[Conjecture 8.4.0.19]{BGT16}.
\section*{Acknowledgments}
We thank Dragos Ghioca and Tom Tucker for many helpful comments and suggestions.  In particular, we are grateful to Dragos Ghioca for mentioning the example following the proof of Corollary \ref{cor:main}.  
 
\end{document}